\documentclass[a4paper,11pt,reqno]{amsart}

\usepackage{amsthm,amsmath,amsfonts,amssymb}
\usepackage{color}

\usepackage{cases}
\usepackage[colorlinks=true,urlcolor=blue,
citecolor=red,linkcolor=blue,linktocpage,pdfpagelabels,
bookmarksnumbered,bookmarksopen]{hyperref}

\usepackage{graphicx}
\usepackage{subfigure}

\usepackage{chngcntr}%%figure numbering according to section
\counterwithin{figure}{section}

\usepackage{pgf,tikz}
\usepackage{mathrsfs}
\usetikzlibrary{arrows}

\newtheorem{thm}{Theorem}[section]
\newtheorem{crit}[thm]{Criterion}
\newtheorem{lem}[thm]{Lemma}
\newtheorem{prop}[thm]{Proposition}

\newtheorem{rem}[thm]{Remark}
\newtheorem{defin}[thm]{Definition}

\newcommand{\R}{{\mathbb{R}}}

\newcommand{\g}{{\gamma}}

\renewcommand{\k}{{\kappa}}
\newcommand{\de}{{\partial}}
\renewcommand{\a}{{\alpha}}
\newcommand{\Om}{\Omega}
\newcommand{\om}{\omega}

\DeclareMathOperator{\RR}{reach}
\DeclareMathOperator{\interior}{int}
\DeclareMathOperator{\dist}{dist}
\DeclareMathOperator{\DIVV}{div}

\begin{document}

\definecolor{ffffff}{rgb}{1.,1.,1.}
\definecolor{cqcqcq}{rgb}{0.75,0.75,0.75}
\definecolor{uuuuuu}{rgb}{0.25,0.25,0.25}

\title[A sufficient criterion to determine planar self-Cheeger sets]{A sufficient criterion to determine\\ planar self-Cheeger sets}
\author{Giorgio Saracco}
\address{Scuola Internazionale Superiore di Studi Avanzati (SISSA), via Bonomea 265, IT--34136 Trieste}
\email{gsaracco@sissa.it}

\thanks{G.~S.~is a member of INdAM and has been partially supported by the INdAM--GNAMPA Project 2019 ``Problemi isoperimetrici in spazi Euclidei e non'' (n.~prot.~U-UFMBAZ-2019-000473 11-03-2019) and the INdAM--GNAMPA Project 2020 ``Problemi isoperimetrici con anisotropie'' (n.~prot.~U-UFMBAZ-2020-000798 15-04-2020).}



\begin{abstract}
We prove a sufficient criterion to determine if a planar set $\Om$ minimizes the prescribed curvature functional $\mathcal{F}_\kappa[E]:=P(E)-\kappa|E|$ amongst $E\subset \Om$. As a special case, we derive a sufficient criterion to determine if $\Om$ is a self-Cheeger set, i.e.~if it minimizes the ratio $P(E)/|E|$ among all of its subsets. As a side effect we provide a way to build self-Cheeger sets.
\end{abstract}
 \hspace{-2cm}
 {
 \begin{minipage}[t]{0.6\linewidth}
 \begin{scriptsize}
 \vspace{-3cm}
 This is a pre-print of an article published in \emph{J. Convex Anal.}. The final authenticated version is available online at: \href{https://www.heldermann.de/JCA/JCA28/JCA283/jca28055.htm}{https://www.heldermann.de/JCA/JCA28/JCA283/jca28055.htm}
 \end{scriptsize}
\end{minipage} 
}

\maketitle

\section{Introduction}

Given an open, bounded set $\Om\subset \R^2$ its \emph{Cheeger constant}, firstly introduced in~\cite{Che70} in a Riemannian setting, is defined as
\begin{equation}\label{eq:Cheeger_const}
h_\Om:=\inf \left\{\, \frac{P(E)}{|E|}: E\subset \Om, |E|>0\,\right\}\,.
\end{equation}
Usually one refers to the task of computing $h_\Om$ and/or of finding sets $E$ attaining the above infimum as to the \emph{Cheeger problem}. Any set $E$ attaining the infimum in~\eqref{eq:Cheeger_const} is called \emph{Cheeger set} of $\Om$; if $\Om$ itself is a minimizer, it is said to be \emph{self-Cheeger}; if $\Om$ is the unique minimizer, it is said to be a \emph{minimal Cheeger}. Notice that any Cheeger set is a nontrivial minimizer of the prescribed curvature functional
\begin{equation}\label{eq:pmc}
\mathcal{F}_\kappa [E] = P(E) -\kappa|E|\,,
\end{equation}
amongst $E\subset \Om$, when the constant $\kappa>0$ is chosen as $\kappa=h_\Om$. The Cheeger problem has drawn a lot of attention because it is intimately tied to many other problems scattered in different fields of mathematics; for instance it is well known that the functional~\eqref{eq:pmc}  admits nontrivial minimizers if and only if $\kappa\ge h_\Om$. The interested reader is referred to~\cite{Leo15, Par11} which are introductory surveys containing basic results and links to other problems, and to~\cite{BP18, Car17, CL19, PS17, Sar18} for further generalizations.\par
In this short note we provide a sufficient criterion to determine if a set $\Om$ is self-Cheeger. This follows from a more general criterion to determine if a set $\Om$ is a minimizer itself of~\eqref{eq:pmc}. In order to state our main result we introduce the following definition of \emph{(strict) interior rolling disk property of radius $R$} for a set $\Om$.

\begin{defin}\label{defin:ird}
We say that a Jordan domain $\Om$ satisfies the \emph{interior rolling disk property of radius $R$} if $\RR(\mathbb{R}^2\setminus \Om)\ge R$. We say that it has the \emph{strict} interior rolling disk property if additionally for all $z\in \de((\mathbb{R}^2\setminus \Om) \oplus B_R)$ no antipodal points of $\de B_R(z)$ lie both on $\de \Om$.
\end{defin}
 
We recall that a Jordan domain is the bounded, open set enclosed by an injective, continuous map $\Phi\colon \mathbb{S}^1 \to \mathbb{R}^2$, which is well defined thanks to the Jordan--Schoenflies Theorem. Furthermore, we recall that a set $A$ has reach $R$ if for all $r<R$ the points in $A\oplus B_r$ have unique projection on $A$, and we refer the reader to the seminal work~\cite{Fed59} and the recent book~\cite{RZ19book}. Roughly speaking, a set $A$ has reach $R$ if it is possible to roll on the exterior of its boundary a ball of radius $R$.
 
%The above definition generalizes one introduced by Chen~\cite[Definition~4.1]{Che80}. First and foremost, he required $\Om$ to have piecewise \emph{smooth} boundary. Second, he was interested only in the \emph{strict interior rolling disk property} which he simply called \emph{interior rolling disk property}.
 %Notice that the above is slightly stronger than requiring an interior ball condition of radius $R$, as for any nonregular point $x$ we require the set $\Om$ to contain \emph{all} balls of radius $R$ whose center lies inside the cone of inwards normals to $\de \Om$ at $x$. Clearly, whenever $\de \Om$ is ${\rm C}^1$ Definition~\ref{defin:ird} corresponds to the usual interior ball condition of radius $R$.

\begin{thm}\label{thm:criterion_pmc}
Let $\Om$ be a Jordan domain. Assume that $\Om$ satisfies the interior rolling disk property of radius $R$ with $R\le |\Om|/P(\Om)$. Then, $\Om$ is a minimizer of $\mathcal{F}_\kappa$ for the choice $\kappa = R^{-1}$. Moreover, if it satisfies the strict property it is the unique minimizer. 
\end{thm}

\begin{rem}\label{rem:nec}
The theorem provides a sufficient but not necessary condition:~consider the bow-tie depicted in Figure~\ref{fig:bowtie}. For suitable choices of $\kappa$ and of the width of the neck, the bow-tie minimizes $\mathcal{F}_\kappa$ in itself but it does not satisfy the criterion, as one can see in~\cite[Example~4.2]{LP16}. If one additionally requires the convexity of $\Om$ the condition essentially becomes necessary. If $\de \Om$ is of class $\mathrm{C}^2$ this is trivial: given $x\in \de \Om$ s.t.~the classic curvature of $\de \Om$ at $x$ is given by $\bar \k$, classic characterization of convexity implies there exists a ball $B$ of radius $\bar \k^{-1}$ entirely contained in $\Om$ such that $x\in \de B$. If $\de \Om$ is less regular, one can use the one-to-one correspondence between convex sets and Radon measures, satisfying a particular distributional inequality, to give a generalized notion of curvature, and the same characterization holds; for a brief discussion of this fact, see for instance~\cite[Section~3]{KL06} and compare with~\cite[Theorem~2]{KL06}.%;~(ii)~the convexity of $\Om$ paired with the self-minimality ensures the smoothness of $\de \Om$. %This means that for convex sets the interior rolling disk property of radius $R$ is equivalent to say that the distributional curvature of $\de \Om$ is bounded from above by $R^{-1}$.
\end{rem}

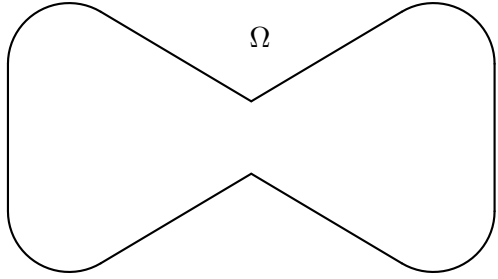
\begin{figure}[t]
\centering
\begin{tikzpicture}[line cap=round,line join=round,>=triangle 45,x=.8cm,y=.8cm]
\clip(-4.2,-2.5) rectangle (4.2,2.5);
\draw [shift={(-2.9932567029285186,1.221899674620759)},line width=.8pt]  plot[domain=1.030376826524313:3.141592653589793,variable=\t]({1.*1.0067432970714816*cos(\t r)+0.*1.0067432970714816*sin(\t r)},{0.*1.0067432970714816*cos(\t r)+1.*1.0067432970714816*sin(\t r)});
\draw [shift={(-2.9932567029285186,-1.221899674620759)},line width=.8pt]  plot[domain=1.030376826524313:3.141592653589793,variable=\t]({1.*1.0067432970714816*cos(\t r)+0.*1.0067432970714816*sin(\t r)},{0.*1.0067432970714816*cos(\t r)+-1.*1.0067432970714816*sin(\t r)});
\draw [shift={(2.9932567029285186,-1.2218996746207595)},line width=.8pt]  plot[domain=1.030376826524313:3.141592653589793,variable=\t]({-1.*1.0067432970714816*cos(\t r)+0.*1.0067432970714816*sin(\t r)},{0.*1.0067432970714816*cos(\t r)+-1.*1.0067432970714816*sin(\t r)});
\draw [shift={(2.9932567029285186,1.2218996746207595)},line width=.8pt]  plot[domain=1.030376826524313:3.141592653589793,variable=\t]({-1.*1.0067432970714816*cos(\t r)+0.*1.0067432970714816*sin(\t r)},{0.*1.0067432970714816*cos(\t r)+1.*1.0067432970714816*sin(\t r)});
\draw [line width=.8pt] (-2.4752915497801284,2.085174929868077)-- (0.,0.6);
\draw [line width=.8pt] (-4.,1.221899674620759)-- (-4.,-1.221899674620759);
\draw [line width=.8pt] (4.,1.2218996746207589)-- (4.,-1.2218996746207595);
\draw [line width=.8pt] (-2.4752915497801284,-2.085174929868077)-- (0.,-0.6);
\draw [line width=.8pt] (2.4752915497801284,2.0851749298680766)-- (0.,0.6);
\draw [line width=.8pt] (2.4752915497801284,-2.0851749298680766)-- (0.,-0.6);
\draw (-0.2,2.0) node[anchor=north west] {$\Omega$};
\end{tikzpicture}
\caption{A bow-tie domain. There is a range of widths of the bow-tie's neck such that the set is a minimizer of the prescribed curvature functional $\mathcal{F}_\kappa$ while it does not satisfy the hypotheses of Theorem~\ref{thm:criterion_pmc}.}\label{fig:bowtie}
\end{figure}

\begin{rem}
It is noteworthy to remark that if a Jordan domain has the interior rolling disk property of radius $R$, then it has the strict interior rolling disk property of radius $r$ for all $r<R$. Thus, Theorem~\ref{thm:criterion_pmc} implies that $\Om$ is the unique minimizer of $\mathcal{F}_\kappa$ with $\k=r^{-1}$ for all choices of $r<R$.
\end{rem}

%\begin{rem}
%Theorem~\ref{thm:criterion_pmc} holds by asking to have the interior rolling disk property of radius $R$ with $R\le h_\Om^{-1}$; clearly one has the trivial inequality $|\Om|/P(\Om)\le h_\Om^{-1}$ since $\Om$ is a viable competitor to produce an upper bound to $h_\Om$. %It will be soon clear why we decided to state it using this ratio.
%\end{rem}

Notice that in the statement of Theorem~\ref{thm:criterion_pmc} one could replace the inequality $R\le |\Om|/P(\Om)$ with the, a-priori, larger one $R\le h_\Om^{-1}$ which is equivalent to the existence of nontrivial minimizers of $\mathcal{F}_{R^{-1}}$. Nevertheless, if $R= |\Om|/P(\Om)$ one ends up proving that $R$ is actually the inverse of the Cheeger constant and that $\Om$ is a self-Cheeger set, thus producing the following criterion.

\begin{crit}\label{crit:main}
Let $\Om$ be a Jordan domain. Assume that it satisfies the interior rolling disk property of radius $R=|\Om|/P(\Om)$. Then, $h_\Om=R^{-1}$ and $\Om$ is self-Cheeger. Moreover, if it satisfies the strict property it is a minimal Cheeger.
\end{crit}

A few remarks are in order. The above criterion is sufficient but not necessary. An example of minimal Cheeger set which does not possess the interior rolling disk property is given by a bow-tie with suitably small neck, depicted in Figure~\ref{fig:bowtie}. Computations for a nonsmoothed out bow-tie are available in~\cite[Example~4.2]{LP16}.\par

As noticed in Remark~\ref{rem:nec}, if the set $\Om$ is convex the criterion is not just sufficient but as well necessary. This is essentially proven in~\cite[Theorem 2]{KL06}, see also~\cite{BCN02, FK02, SZ97}. Yet, there are nonconvex sets which are covered by Criterion~\ref{crit:main}. Aside from the bow-tie with sufficiently large neck previously discussed, other examples are provided by suitable strips that were shown to be self-Cheeger in~\cite{KP11, LP16}, under some technical assumption on their length, which can now be dropped in view of Criterion~\ref{crit:main} or the results in~\cite{LNS17}. More in general, taken any convex self-Cheeger set $\Om$ one can add a ``tendril'' of suitable width producing a set which is covered by the criterion. This is exemplified by the Pinocchio example shown in Figure~\ref{fig:pinocchio}. One can in principle add as many tendrils as s/he wishes; for instance taking two directly opposed to each other rather than Pinocchio's nose one produces a ``cloud'' or a ``Dumbo'' set which is still covered by the criterion. Computations for these sets are available in~\cite[Example~4.6 and~4.7]{LP16}.

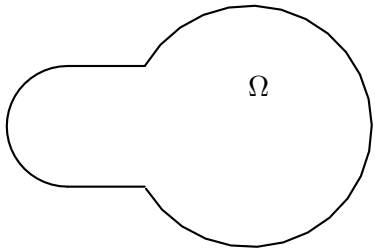
\begin{figure}[t]
\centering
\begin{tikzpicture}[line cap=round,line join=round,>=triangle 45,x=.8cm,y=.8cm]
\clip(-4.2,-2.2) rectangle (2.2,2.2);
\draw [shift={(-3.,0.)},line width=.8pt]  plot[domain=1.5707963267948966:4.71238898038469,variable=\t]({1.*1.*cos(\t r)+0.*1.*sin(\t r)},{0.*1.*cos(\t r)+1.*1.*sin(\t r)});
\draw [shift={(0.,0.)},line width=.8pt]  plot[domain=-2.5977601972876707:2.6179938779914944,variable=\t]({1.*2.*cos(\t r)+0.*2.*sin(\t r)},{0.*2.*cos(\t r)+1.*2.*sin(\t r)});
\draw [line width=.8pt] (-3.,1.)-- (-1.7320508075688772,1.);
\draw [line width=.8pt] (-3.,-1.)-- (-1.7320508075688772,-1.);
\draw (-0.2,1.0) node[anchor=north west] {$\Omega$};
\end{tikzpicture}
\caption{The ``Pinocchio'' set: a nonconvex self-Cheeger set that is covered by the criterion. The ``face'' is a ball of radius $1$, while the ``nose'' is a ``tendril'' of suitable width which can be made as long as one wishes.}\label{fig:pinocchio}
\end{figure}

A weaker version of the ``strict'' part of Criterion~\ref{crit:main} is already present in the literature by combining two theorems centered around the capillarity problem; yet up to our knowledge it has never been presented in the terminology of Cheeger sets and such an elegant criterion is missing in the two widespread surveys~\cite{Leo15, Par11} on Cheeger sets. Hence, it is of interest making it readily available to the ``Cheeger community''. More precisely, in~\cite[Theorem~4.1]{Che80} Chen proves that a \emph{piecewise smooth} set $\Om$ possessing the strict interior rolling disk property\footnote{We warn the interested reader that in~\cite{Che80} Chen defines as interior rolling disk property what we here call \emph{strict} interior rolling disk property.} of radius $|\Om|/P(\Om)$ is such that the nonlinear PDE 
\begin{equation}\label{eq:cap}
\begin{cases}
\DIVV\left( \frac{\nabla u}{\sqrt{1+|\nabla u|^2}}\right) = \frac{P(\Om)}{|\Om|}\cos \a\,, \qquad &\text{in $\Om$},\\
\frac{\nabla u \cdot \nu_\Om}{\sqrt{1+|\nabla u|^2}}  = \cos \a\,, &\text{on $\de \Om$},
\end{cases}
\end{equation}
has solutions for all choices of angles $\a$ in the range $[0, \pi/2]$. In~\cite[Theorem~4.7 and~5.1]{LS18a} (see also the seminal paper~\cite{Giu78}) it is proved that existence of solutions of~\eqref{eq:cap} for the choice $\a=0$ is equivalent to say that (a piecewise Lipschitz) $\Om$ is a minimal Cheeger set. Hence, the ``strict'' part of Criterion~\ref{crit:main}  for piecewise smooth sets follows by combining these two results. As a consequence of Criterion~\ref{crit:main} one also gets existence of solutions of~\eqref{eq:cap} in the nonstrict case for all $\a \in (0, \pi/2]$ (we refer the interested reader to the comprehensive treatise~\cite{Fin86book}).\par

Thanks to very recent results which we recollect in Section~\ref{sec:prel} we are able to give a very short proof of Theorem~\ref{thm:criterion_pmc} and of Criterion~\ref{crit:main}, which are contained in Section~\ref{sec:proofs}. As a side result of the main theorem and of Steiner's formulae, we are able to provide a way to construct Cheeger sets as stated in the following proposition, whose proof is as well contained in Section~\ref{sec:proofs}.

\begin{prop}\label{prop:main}
Let $\om \subset \R^2$ be a closed, simply connected set such that $|\om|=\pi R^2$ and $\RR(\om) > R$. Then, the Minkowski sum $\Om = \om\oplus B_R$ is self-Cheeger. Moreover, if $\om = \overline{\interior(\om)}$, then $\Om$ is a minimal Cheeger.
\end{prop}

%{\color{blue}
%\begin{rem}
%The above proposition can be easily extended to build sets $\Om$ that are minimizers of $\mathcal{F}_\kappa$ in $\Om$ itself for suitable choices of $\kappa$. Specifically, if $\om \subset \R^2$ is a closed, simply connected set such that $|\om|\ge\pi R^2$ with $\RR(\om)> R$, then the Minkowski sum $\Om = \om \oplus B_R$ minimizes~\eqref{eq:pmc} for all $\k \ge R^{-1}$.
%\end{rem}
%}

\section{Preliminaries}\label{sec:prel}

We recall the definition of \emph{no necks of radius $R$} for a planar domain $\Om$.

\begin{defin}[Definition~1.2 of~\cite{LNS17}]\label{def:nonecks}
A Jordan domain $\Om$ has \emph{no necks of radius $R$} if taken any two balls $B_R^0$ and $B_R^1$ of radius $R$ entirely contained in $\Om$ there exists a continuous curve $\g\colon [0,1]\to \Om$ s.t.~$B_R(\g(0))=B^0_R$, $B_R(\g(1))=B^1_R$ and $B_R(\g(t)) \subset \Om$ for all times $t\in[0,1]$.
\end{defin}

\begin{rem}\label{rem:reg}
One can suppose the curve $\g$ to be of class $\mathrm{C}^{1,1}$ with curvature bounded by $1/R$ thanks to~\cite[Theorem 1.8]{LNS17}, see also~\cite[Lemma~5.1]{LNS17} combined  with~\cite[Theorem~1.2,~1.3]{Lyt05} or~\cite[Remark~6.7]{RZ17}. Moreover, asking a set $\Om$ to have no necks of radius $R$ is equivalent to ask that the inner parallel set $\Om_R := \{\,x\in \Om: \dist(x;\de\Om) \ge R\,\}$ is path-connected.
\end{rem}

Whenever a Jordan domain $\Om$ has no necks of radius $\k^{-1}\le h_\Om^{-1}$, Leonardi and myself proved a structure theorem for minimizers of the prescribed curvature functional $\mathcal{F}_\k [E] := P(E)-\k|E|$ amongst $E\subset \Om$ in~\cite{LS19}. Such a theorem extends~\cite[Theorem~1.4]{LNS17} obtained jointly with Neumayer valid in the limit case $\k=h_\Om$. Since the class of minimizers is closed under countable unions, one can define a (unique) \emph{maximal} minimizer, of which a full geometric characterization is available thanks to the abovementioned results. In particular, by defining the interior parallel set of $\Om$ at distance $r$ as
\begin{equation}\label{eq:inner_parallel}
\Om_r :=  \{\,x\in \Om: \dist(x;\de\Om) \ge r\,\}
\end{equation}
one has that the maximal minimizer of $\mathcal{F}_\kappa$ is given by the Minkowski sum $\Om_R \oplus B_R$, where $R=\k^{-1}$. This result is essentially sharp as can be seen by the examples contained in~\cite{LNS17, LS18b}. For the sake of completeness we recall below the full statement.

\begin{thm}[Theorem~2.3 of~\cite{LS19}]\label{thm:LNS}
Let $\Om \subset \R^2$ be a Jordan domain with $|\de \Om|=0$, and let $\kappa\ge h_\Om$ be fixed. If $\Om$ has no necks of radius $R=\kappa^{-1}$, then the maximal minimizer of the prescribed curvature functional $\mathcal{F}_\k$ is given by $\Om_R \oplus B_R$. Moreover, if $\Om_R = \overline{\interior(\Om_R)}$, it is the unique minimizer.
\end{thm}

\section{Proofs of the results}\label{sec:proofs}

In view of the structure theorem of minimizers of $\mathcal{F}_\kappa$ provided by Theorem~\ref{thm:LNS}, the proofs of Theorem~\ref{thm:criterion_pmc} and of Criterion~\ref{crit:main} boil down to show that $\Om$ satisfying the interior rolling disk property of radius $R$ is equivalent to say that $\Om$ has no necks of radius $R$ and it agrees with the Minkowski sum $\Om_R\oplus B_R$. This is exactly what we show in the next lemma.

\begin{lem}\label{lem:equi}
Let $\Om$ be a Jordan domain. Then, it satisfies the interior rolling disk property of radius $R$ if and only if it has no necks of radius $R$ and the set equality $\Om = \Om_R\oplus B_R$ holds. Moreover, if it satisfies the strict property of radius $R$ one has the set equality $\Om_R = \overline{\interior(\Om_R)}$.
\end{lem}

\begin{proof}
By~\cite[Lemma 4.8]{RZ19book} %{\color{blue}
if
%}
$\RR(\mathbb{R}^2\setminus \Om)\ge R$, then $\mathbb{R}^2\setminus \Om$ is (morphologically) closed w.r.t.~$B_R$, i.e.~$\mathbb{R}^2\setminus \Om=((\mathbb{R}^2\setminus \Om)\oplus B_R)\ominus B_R$. Moreover, a set $A$ is (morphologically) closed w.r.t.~$B_R$ if and only if its complement set is (morphologically) open, i.e.~$\Om=(\Om\ominus B_R)\oplus B_R=\Om_R \oplus B_R$. We are left with showing that $\Om_R$ is path-connected. This is a consequence of $\Om$ being a Jordan domain. As $\mathbb{R}^2\setminus \Om$ has reach $R$, for any $0<r<R$ its $r$-offset $(\mathbb{R}^2\setminus \Om) \oplus B_r$ has $\mathrm{C}^{1,1}$ boundary~\cite[Corollary~4.22]{RZ19book}. Moreover, the projection
\[
\Pi_{\de \Om}:\big( (\mathbb{R}^2\setminus \Om) \oplus B_r \big) \setminus (\mathbb{R}^2\setminus \Om) \to \de \Om
\]
is a deformation retract, thus $\de \Om$ and the $\mathrm{C}^{1,1}$ boundary of $(\mathbb{R}^2\setminus \Om) \oplus B_r$ are homotopic, see~\cite[Lemma~4.52]{RZ19book}. Hence, this $\mathrm{C}^{1,1}$ boundary has only one connected component and it is the image of a Jordan curve $\gamma_r$. Taken now any two points $x_0, x_1 \in \Om_R$, we let $w_0, w_1$ be any of their projections on $\de \Om$. Denoted by $\overline{x_iw_i}$ the segment with endpoints $x_i, w_i$, we set $z_i$ to be the point on $\de \Om_R \cap \overline{w_ix_i}$ and $z^r_i$ be the point on $\text{Im}(\gamma_r) \cap \overline{w_ix_i}$. It is now easy to provide a continuous curve, possibly not simple, from $z_0$ to $z_1$ in $\de \Om_R$. Restricting $\gamma_{r}$ to $[\gamma_{r}^{-1}(z^r_0),\gamma_{r}^{-1}(z^r_1)]$, and reparametrizing it, this is given by
\[
\tilde \gamma(t) : = \gamma_{r}(t) +(R-r) \frac{\dot \gamma_r^\perp(t)}{|\dot\gamma_r(t)|}\,,
\]
up to changing orientation. A concatenation of the segments $\overline{x_iz_i}$ and of $\tilde \gamma$ gives the desired curve.
\par
%{\color{blue}
We now show the opposite direction. By~\cite[Lemma~5.1]{LNS17}, $\RR(\Om_R)\ge R$, thus by~\cite[Lemma~4.8]{RZ19book} $\Om_R$ is (morphologically) closed w.r.t.~$B_R$, and thus its complement set is (morphologically) open, i.e.
\begin{equation}\label{eq:set_eq_R}
\mathbb{R}^2\setminus \Om_R = \Big((\mathbb{R}^2\setminus\Omega_R)\ominus B_R\Big)\oplus B_R\,.
\end{equation}
As $A^c \ominus B_R = (A\oplus B_R)^c$, from~\eqref{eq:set_eq_R} it follows that
\begin{align*}
\mathbb{R}^2\setminus \Om_R&=\Big((\mathbb{R}^2\setminus\Omega_R)\ominus B_R\Big)\oplus B_R = (\Omega_R \oplus B_R)^c \oplus B_R \\
&= \Omega^c \oplus B_R = (\mathbb{R}^2 \setminus \Omega)\oplus B_R\,.
\end{align*}
Therefore, we have
\begin{equation}\label{eq:set_eq_R_2}
\Big((\mathbb{R}^2 \setminus \Omega)\oplus B_R\Big) \setminus (\mathbb{R}^2 \setminus \Omega) = \Big(\mathbb{R}^2\setminus \Om_R\Big)  \setminus \Big((\mathbb{R}^2 \setminus \Omega)\Big) = \Omega \setminus \Om_R\,.
\end{equation}
As $\RR(\Om_R)\ge R$, all $x\in \Om \setminus \Om_R$ have unique projection $z_x$ on $\de \Om_R$. By~\eqref{eq:set_eq_R_2}, our claim consists in showing that any of these $x$ has as well unique projection on $\de \Om$. This is straightforward as $|x-z_x|>0$, $B_R(z_x) \supset B_{\dist(x; \de \Om)}(x)$ and thus $\de B_R(z) \cap \de B_{\dist(x; \de \Om)}(x)$ contains at most one point. Hence, it contains \emph{exactly} one.
%}
\par
We are left to show the last part of the claim;~suppose that $\Om$ has the strict property and let by contradiction $\Om_R \setminus \overline{\interior(\Om_R)}\neq \emptyset$. By~\cite[Proposition~2.1]{LS19}, this set difference consists of $\mathrm{C}^{1,1}$ curves. By~\cite[Remark~4.1]{LS19} any point $x$ on these curves has two antipodal projections on $\de \Om$. This yields a contradiction with assuming the strict property to hold.
\end{proof}

\begin{rem}
Notice that the second part of Lemma~\ref{lem:equi} is not an ``if and only if''. There exist sets $\Om$ with no necks of radius $R$ such that their inner parallel set $\Om_R$ agrees with the closure of its interior, $\Om$ equals the Minkowski sum $\Om_R\oplus B_R$ but for which the strict interior rolling disk property of radius $R$ does not hold. An example is depicted in Figure~\ref{fig:double_bubble}. Notice that the ball with antipodal points on $\de \Om$ is centered on a point that, if removed, would disconnect $\Om_R$.
\end{rem}

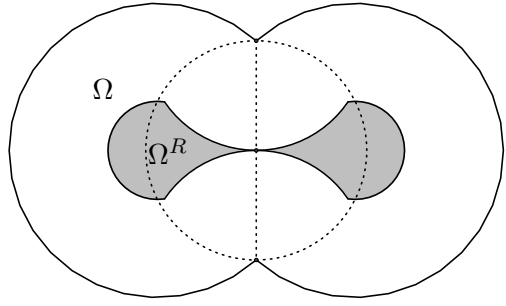
\begin{figure}[t]
\centering
\begin{tikzpicture}[line cap=round,line join=round,>=triangle 45,x=1.0cm,y=1.0cm, scale=.65]
\clip(-5.2,-3.3) rectangle (5.2,3.3);
%%%POLIGONO
\fill[line width=0.pt,color=cqcqcq,fill=cqcqcq,fill opacity=1.0] (-1.6015221281203638,0.6755786718200547) -- (-1.6029105656623637,-0.677004887535833) -- (1.606667296651344,-0.6808766073234171) -- (1.5785820070773684,0.6523764523183266) -- cycle;

\draw [line width=0.6pt,color=cqcqcq,fill=cqcqcq,fill opacity=1.0] (-2.,0.) circle (1.cm);
\draw [line width=0.6pt,color=cqcqcq,fill=cqcqcq,fill opacity=1.0] (2.,0.) circle (1.cm);

\draw [line width=0.6pt,color=ffffff,fill=ffffff,fill opacity=1.0] (0.,-2.23606797749979) circle (2.23606797749979cm);
\draw [line width=0.6pt,color=ffffff,fill=ffffff,fill opacity=1.0] (0.,2.23606797749979) circle (2.23606797749979cm);

\draw [shift={(-2.,0.)},line width=0.6pt]  plot[domain=0.8410686705679302:5.442116636611656,variable=\t]({1.*3.*cos(\t r)+0.*3.*sin(\t r)},{0.*3.*cos(\t r)+1.*3.*sin(\t r)});
\draw [shift={(2.,0.)},line width=0.6pt]  plot[domain=0.8410686705679302:5.442116636611656,variable=\t]({-1.*3.*cos(\t r)+0.*3.*sin(\t r)},{0.*3.*cos(\t r)+1.*3.*sin(\t r)});
\draw [shift={(-2.,0.)},line width=0.6pt]  plot[domain=1.426737205580549:4.856448101599037,variable=\t]({1.*0.9999811971695507*cos(\t r)+0.*0.9999811971695507*sin(\t r)},{0.*0.9999811971695507*cos(\t r)+1.*0.9999811971695507*sin(\t r)});
\draw [shift={(2.,0.)},line width=0.6pt]  plot[domain=1.426737205580549:4.856448101599037,variable=\t]({-1.*0.9999811971695507*cos(\t r)+0.*0.9999811971695507*sin(\t r)},{0.*0.9999811971695507*cos(\t r)+1.*0.9999811971695507*sin(\t r)});
\draw [shift={(0.,-2.23606797749979)},line width=0.6pt]  plot[domain=0.5912836787256022:2.550308974864191,variable=\t]({1.*2.23606797749979*cos(\t r)+0.*2.23606797749979*sin(\t r)},{0.*2.23606797749979*cos(\t r)+1.*2.23606797749979*sin(\t r)});
\draw [shift={(0.,2.23606797749979)},line width=0.6pt]  plot[domain=0.5912836787256022:2.550308974864191,variable=\t]({1.*2.23606797749979*cos(\t r)+0.*2.23606797749979*sin(\t r)},{0.*2.23606797749979*cos(\t r)+-1.*2.23606797749979*sin(\t r)});
\draw (-3.521404093983697,1.664831777956845) node[anchor=north west] {$\Omega$};
\draw (-2.4,0.45) node[anchor=north west] {$\Omega^R$};
\draw [line width=0.6pt,dotted] (0.,0.) circle (2.23606797749979cm);
\draw [line width=0.6pt,dotted] (0.,2.23606797749979)-- (0.,-2.23606797749979);
\begin{scriptsize}
\draw [fill=uuuuuu] (0.,0.) circle (1.0pt);
\draw [fill=uuuuuu] (0.,2.23606797749979) circle (1.0pt);
\draw [fill=uuuuuu] (0.,-2.23606797749979) circle (1.0pt);
\end{scriptsize}
\end{tikzpicture}
\caption{A set $\Om$ with no necks of radius $R$ such that ${\Om=\Om_R \oplus B_R}$, $\Om_R = \overline{\interior{(\Om_R)}}$ but for which the strict interior rolling disk property of radius $R$ does not hold.}\label{fig:double_bubble}
\end{figure}

\begin{proof}[Proof of Theorem~\ref{thm:criterion_pmc} and of Criterion~\ref{crit:main}]
First, notice that the $2$ dimensional Lebesgue measure of $\de \Om$ has zero measure, i.e.~$|\de \Om|=0$, since one has $\RR(\mathbb{R}^2\setminus \Om)>0$, see~\cite[Theorem 6.1~(v),~(vii) and Theorem~6.2~(ii)]{DZ98}. Second, notice that as one can use $\Om$ as competitor for estimating the Cheeger constant, one has the upper bound $h_\Om\le R^{-1}$. In virtue of Lemma~\ref{lem:equi} and of Theorem~\ref{thm:LNS} we immediately find that $\Om$ is a minimizer of
\[
\mathcal{F}_\k [E] = P(E)-\kappa|E|,
\]
for $\kappa = R^{-1}$. If the strict property given in Definition~\ref{defin:ird} holds, i.e.~none of the interior rolling disks is such that two antipodal points of the boundary lie on $\de \Om$, uniqueness follows as well from Lemma~\ref{lem:equi}.\par
Suppose now that $R=|\Om|/P(\Om)$ and argue by contradiction letting $R^{-1}>h_\Om$. Notice that the minimum of $\mathcal{F}_{R^{-1}}$ is strictly negative. Indeed, taken a Cheeger set $E$ of $\Om$, which are well known to exist, it is immediate to check that
\[
P(\Om)-R^{-1}|\Om| = \min \mathcal{F}_{R^{-1}}  \le P(E)-R^{-1} |E| < P(E)-h_\Om|E| = 0.
\]
Then,
\begin{align*}
R < \frac{|\Om|}{P(\Om)}.
\end{align*}
This immediately produces a contradiction, which concludes the proof.
\end{proof}

\begin{proof}[Proof of Proposition~\ref{prop:main}]
Consider the set $\Om = \om \oplus B_R$. Since $\om$ has reach greater than $R$, one has that $\Om_R = \om$. Thanks to the regularizing effect of the Minkowski sum of a set with reach strictly greater than $R$ with $B_R$, see~\cite[Corollary~4.22]{RZ19book}, we have that $\Omega$ is $\mathrm{C}^{1,1}$. Moreover, through the projection $\Pi_{\de \omega} : \Omega\setminus \omega \to \omega$, one can define a deformation retract between $\Omega$ and $\omega$, and thus they are homotopic, see~\cite[Lemma~4.52]{RZ19book}. Hence,  $\Omega$ is connected and simply connected, thus a Jordan domain.
\par
In virtue of the validity of Steiner's formulae (see~\cite[Section~2.3]{LNS17} or the more general~\cite[Section~4.5]{RZ19book}), one has
\begin{align*}
|\Om| = |\om| + R\mathcal{M}_o(\om) + \pi R^2, && P(\Om) = \mathcal{M}_o(\om) + 2\pi R,
\end{align*}
where $\mathcal{M}_o$ denotes the outer Minkowski content. On the one hand by hypothesis the equality $|\om|=\pi R^2$ holds, thus the ratio $|\Om|/P(\Om)$ equals $R$. On the other hand, the construction itself implies that $\Om$ possesses the interior rolling disk property of radius $R$. By Criterion~\ref{crit:main}, the claim immediately follows.
\end{proof}

\begin{rem}
Note that we need the strict inequality $\RR(\omega)>R$ in Proposition~\ref{prop:main}. Indeed, a key tool in the proof is Steiner's formulae which, for sets of reach $\rho$, are valid up to $\rho$ but not $\rho$ itself.
\end{rem}

\bibliographystyle{plainurl}

\bibliography{self_Cheeger_criterion}

\end{document}